\documentclass[pdflatex,sn-mathphys-num]{sn-jnl}

\usepackage{graphicx}%
\usepackage{multirow}%
\usepackage{amsmath,amssymb,amsfonts}%
\usepackage{amsthm}%
\usepackage{mathrsfs}%
\usepackage[title]{appendix}%
\usepackage{xcolor}%
\usepackage{textcomp}%
\usepackage{manyfoot}%
\usepackage{booktabs}%
\usepackage{algorithm}%
\usepackage{algorithmicx}%
\usepackage{algpseudocode}%
\usepackage{listings}%
\usepackage{tikz,pgfplots}%
\usepackage{siunitx}%

\newcommand{\set}[1]{\left\{#1\right\}}
\newcommand{\abs}[1]{\left|#1\right|}
\newcommand{\qand}{\quad\text{and}\quad}
\newcommand{\eps}{\varepsilon}
\newcommand{\p}{\partial}

\newcommand{\mF}{\mathbf{F}}
\newcommand{\mG}{\mathbf{G}}
\newcommand{\mr}{\mathbf{r}}
\newcommand{\mt}{\mathbf{t}}
\newcommand{\mx}{\mathbf{x}}

\newcommand{\mz}{\mathbf{z}}
\newcommand{\vt}{\boldsymbol{\theta}}
\newcommand{\vv}{\boldsymbol{\vartheta}}

\DeclareMathOperator*{\bfm}{BFM}
\DeclareMathOperator*{\bfn}{BFN}
\DeclareMathOperator*{\inc}{inc}
\DeclareMathOperator*{\area}{area}
\DeclareMathOperator*{\scat}{scat}

\theoremstyle{thmstyleone}%
\newtheorem{theorem}{Theorem}
\newtheorem{corollary}{Corollary}
\theoremstyle{thmstyletwo}%
\newtheorem{example}{Example}%
\newtheorem{remark}{Remark}%

\theoremstyle{thmstylethree}%

\begin{document}

\title[Mathematical and experimental validation of the bifocusing method tailored for bistatic measurement]{Mathematical and experimental validation of the bifocusing method tailored for bistatic measurement}

\author{\fnm{Won-Kwang} \sur{Park}}\email{parkwk@kookmin.ac.kr}
\affil{\orgdiv{Department of Information Security, Cryptology, and Mathematics}, \orgname{Kookmin University}, \orgaddress{\city{Seoul}, \postcode{02707},  \country{Republic of Korea}}}


\abstract{In this paper, we design a bifocusing-based imaging strategy for the rapid identification of small penetrable dielectric inhomogeneities within a two-dimensional bistatic measurement setup. To address the applicability and limitation, we carefully explore the mathematical structure of the indicator function by establishing a relationship involving the infinite series of Bessel functions, the material characteristics, and the bistatic angle. Through this theoretical result, we rigorously verify that the imaging resolution degrades as the bistatic angle approaches $\SI{180}{\degree}$, and specifically, that target identification becomes impossible when the bistatic angle is $\SI{180}{\degree}$. Conversely, relatively high-resolution results are obtained when the bistatic angle is close to $\SI{0}{\degree}$. The theoretical findings are validated through numerical simulations using the Fresnel experimental dataset, which confirm the applicability and limitations of the proposed method for both dielectric and metallic objects.}

\keywords{Inverse scattering problem, Bifocusing method, Bistatic measurement, Bessel function, Numerical simulations}



\maketitle

\section{Introduction}\label{sec:1}

The inverse scattering problem, which aims to reconstruct the location and shape of unknown targets from measured scattered field data, is a fundamental subject in various engineering fields \cite{C8,P-Book}, including medical imaging \cite{A1,A2}, non-destructive testing \cite{TLRD,WWXHSZLF}, and radar imaging \cite{C7,P-Book2}. In particular, the bistatic measurement configuration, where the transmitter and receiver are separated by a fixed or varying angle, has garnered significant attention. Unlike monostatic setups where the transmitter and receiver are collocated, bistatic configurations offer distinct advantages, such as reduced interference between transmitted and received signals and the ability to capture richer scattering information through diverse observation angles \cite{C5,G4}. Consequently, developing efficient imaging algorithms tailored for bistatic setups is a crucial and challenging task in practical applications such as ground-penetrating radar \cite{LFNA,SL2}, through-wall imaging \cite{B6,CPPPDT}, and so on.

To address the inverse scattering problem in bistatic or multi-static measurement environments, various quantitative approaches based on Newton-type iterative schemes have been successfully employed. For instance, the Distorted Born Iterative Method to reconstruct high-contrast dielectric profiles \cite{CW}, the Contrast Source Inversion (CSI) method for reconstructing the complex index of
refraction of a bounded object \cite{BK1}, Gauss-Newton method \cite{ML2}, Newton-Kantorovich iteration scheme to retrieve the complex permittivity of inhomogeneous lossy dielectric objects \cite{JPH}, level-set method for shape reconstruction from transverse magnetic(TM) and transverse electric(TE) polarized waves \cite{DL}. However, although these iterative methods offer high-resolution quantitative images, they are computationally expensive and suffer from intrinsic non-linearity and ill-posedness (see \cite{EG} for instance). Note that, based on \cite{KSY,PL4}, the iterative-based algorithms must be applied with a good initial guess. If not, one will encounter various issues e.g., non-convergence, obtaining a local minimizer, and tremendous computational times. 

To overcome these challenges, various non-iterative qualitative methods have been developed for the fast identification of inhomogeneities. Specifically, in the context of full- or limited-aperture measurement environments, several qualitative approaches have proven effective e.g., direct and orthogonality sampling methods \cite{IJZ1,P1}, MUltiple SIgnal Classification (MUSIC) algorithm \cite{D1,P-MUSIC11}, factorization and linear sampling methods \cite{LLP,MPS1}, and variational approaches \cite{AIM,R2}. Among these approaches, the Bifocusing Method (BFM) has attracted attention due to its simplicity and effectiveness in focusing the scattered energy onto the target locations such as damage detection of concrete void \cite{KJFF}, ultra-wide-band tomographic and radar imaging \cite{JBRBTFC,JTMCRBC}, localization of small objects in limited-aperture inverse scattering problem \cite{KLPS2,KP4}, and identification of the existence and outline shape of dielectric targets in microwave imaging \cite{KLPS1,KPS,SP2}. However, most studies have primarily relied on multi-static measurement environment so, application and rigorous mathematical analysis of the BFM indicator function in a single-frequency bistatic measurement setup remains insufficient. Specifically, the impact of the bistatic angle---the fixed angle between the transmitter and receiver---on the imaging performance requires a thorough theoretical investigation to ensure reliable detection.

In this paper, we consider the application of the BFM for identifying small dielectric inhomogeneities in a bistatic measurement configuration. The main contribution of this paper is the derivation of a closed-form mathematical structure of the indicator function by means of the infinite series of the Bessel function, material properties of the inhomogeneities, and bistatic angle. This result reveals how the resolution and detection capability are intrinsically dependent on the bistatic angle $\alpha$. We theoretically demonstrate that the resolution is maximized when $\alpha$ approaches $\SI{0}{\degree}$ (quasi-monostatic) and significantly degrades as $\alpha$ approaches $\SI{180}{\degree}$, eventually leading to a complete inappropriate imaging result at $\alpha=\SI{180}{\degree}$.

The remainder of this paper is organized as follows. In Section \ref{sec:2}, we formulate the direct scattering problem and introduce the BFM-based indicator function. Section \ref{sec:3} presents the theoretical analysis of the indicator function, establishing its relationship with Bessel functions and discussing its properties regarding the bistatic angle. In Section \ref{sec:4}, we present numerical simulation results using the Fresnel experimental dataset to validate our theoretical claims and demonstrate the method's applicability to both single and multiple targets, as well as metallic objects. Finally, concluding remarks are provided in Section \ref{sec:5}.

\section{Problem Formulation and Imaging Methodology}\label{sec:2}

In this section, we formulate the two-dimensional inverse scattering problem for identifying small dielectric inhomogeneities and introduce a direct sampling-type indicator function within a bistatic measurement configuration.

\subsection{Direct Scattering Problem and Representation Formula}

Let $\Omega\subset\mathbb{R}^2$ denote a bounded homogeneous background region characterized by vacuum permittivity $\eps_0$ and permeability $\mu_0$. We consider the time-harmonic regime with an angular frequency $\omega$, where the background wavenumber is given by $k = \omega\sqrt{\eps_0\mu_0}$. The region $\Omega$ contains a finite collection of distinct, small dielectric inhomogeneities with smooth boundaries, denoted by $\{D_m\}_{m=1}^M$.

Each inhomogeneity $D_m$ is modeled as a small inclusion characterized by its permittivity $\eps_m>\eps_0$ at the frequency of operation. We assume that the inhomogeneities are well-separated, satisfying the condition $|\mathbf{x}_m - \mathbf{x}_{m'}| \gg \lambda$, where $\lambda$ is the wavelength, $\mz_m\in D_m$, and $\mz_{m'}\in D_{m'}$. The spatial distribution of the permittivity $\eps(\mathbf{x})$ in the entire domain is defined as a piecewise constant function:
\begin{equation}
\eps(\mathbf{x})=\left\{\begin{array}{rl}
\smallskip\eps_m,&\mathbf{x}\in D_m, \\
\eps_0,&\mathbf{x}\in\Omega\backslash\overline{D},
\end{array}\right.\quad\text{where}\quad D=\bigcup_{m=1}^{S}D_m.
\end{equation}

The interrogation of the domain is performed using electromagnetic waves generated by a transmitter: we denote $u_{\inc}(\mathbf{x},\mathbf{t})$, in the absence of inhomogeneities, as the $z$-component of the incident electric field generated by a point source $\mathcal{T}$ located at $\mathbf{t}$ that is given by
\[u_{\inc}(\mathbf{x},\mathbf{t})=G(\mathbf{x},\mathbf{t})=-\frac{i}{4}H_0^{(1)}(k|\mathbf{x}-\mathbf{t}|),\]
where $H_0^{(1)}$ is the Hankel function of the first kind of order zero. Let $u(\mathbf{x},\mathbf{t})$ denote the $z$-component of the total electric field at a location $\mathbf{x}\in\Omega$ due to a point source $\mathcal{T}$ located at $\mathbf{t}$. Note that the total field satisfies the scalar Helmholtz equation: for every $m=1,2,\ldots,M$,
\begin{equation}
\left\{\begin{array}{rcl}
\smallskip\Delta u(\mathbf{x},\mathbf{t})+\omega^2\eps(\mathbf{x})\mu_0 u(\mathbf{x},\mathbf{t})=0&\text{in}&\Omega,\\
u(\mathbf{x},\mathbf{t})|_--u(\mathbf{x}, \mathbf{t})|_+=0&\text{on}&\p D_m.
\end{array}\right.
\end{equation}
With this, we denote $u_{\scat}(\mx,\mt)=u(\mx,\vt)-u_{\inc}(\mx,\vt)$ as the scattered field that satisfies the Sommerfeld radiation condition
\[\lim_{|\mx|\to\infty}\sqrt{|\mx|}\left(\frac{\p u_{\scat}(\mx,\mt)}{\p|\mx|}-iku_{\scat}(\mx,\mt)\right)=0\quad\text{uniformly in all directions}\quad\frac{\mx}{|\mx|}.\]

The main purpose of this paper is to retrieve $\mz\in D$ from measured scattered field data $u_{\scat}(\mr,\mt)$ measured at a receiver $\mathcal{R}$ located at $\mr$. To this end, one needs an expression of the $u_{\scat}(\mr,\mt)$ in the presence of $D$. Based on \cite{BCS}, it can be represented as the following integral equation formula:
\begin{equation}\label{ScatteredField}
u_{\scat}(\mr,\mt)\approx k^2\int_D\left(\frac{\eps(\mz)-\eps_0}{\eps_0\mu_0}\right)G(\mr,\mz)G(\mz,\mt)d\mz.
\end{equation}

\subsection{Bifocusing-Based Indicator Function}
Assume that $N$ distinct scattered field data $u_{\scat}(\mr_n,\mt_n)$, $n=1,2,\ldots,N$ were collected under a bistatic measurement setup with a fixed bistatic angle $\alpha$ between the moving transmitter and receiver (see Fig. \ref{setup} for an illustration). Throughout this paper, we set
\[\mt_n=|\mt_n|\vt_n=T(\cos\theta_n,\sin\theta_n)\qand\mr_n=|\mr_n|\vv_n=R\big(\cos(\theta_n+\alpha),\sin(\theta_n+\alpha)\big),\]
where $0\leq\alpha\leq\pi$ and 
\[\theta_n=(n-1)\triangle\theta=\frac{2\pi(n-1)}{N},\quad n=1,2,\ldots,N.\]

\begin{figure}[h]
\centering
\begin{tikzpicture}[scale=3.5]
\draw[gray,very thin] (1,0) arc (0:360:1);

\draw[black,dashed,-] (0,0) -- (1,0);
\draw[red,very thick,solid,stealth-] (0.6,0) -- (0.95,0);

\draw[black,dashed,-] (0,0) -- ({cos(140)},{sin(140});
\draw[green,very thick,solid,-stealth] ({0.6*cos(140)},{0.6*sin(140)}) --  ({0.96*cos(140)},{0.96*sin(140)});

\foreach \alpha in {-60,-40,...,60}
{\draw[red!30,fill=red!30] ({cos(\alpha)},{sin(\alpha)}) circle (0.03cm);
\draw[black!30,fill=black!30] ({cos(\alpha)},{sin(\alpha)}) circle (0.01cm);}

\node at (1.1,0) {$\mathcal{T}_n$};
\node at ({1.12*cos(140)},{1.12*sin(140)}) {$\mathcal{R}_n$};

\draw[red,fill=red] ({cos(0)},{sin(0)}) circle (0.03cm);
\draw[black,fill=black] ({cos(0)},{sin(0)}) circle (0.01cm);

\foreach \alpha in {120,140,...,240}
{\draw[green!30,fill=green!30] ({cos(\alpha)},{sin(\alpha)}) circle (0.03cm);
\draw[black!30,fill=black!30] ({cos(\alpha)},{sin(\alpha)}) circle (0.01cm);}

\draw[green,fill=green] ({cos(140)},{sin(140)}) circle (0.03cm);
\draw[black,fill=black] ({cos(140)},{sin(140)}) circle (0.01cm);

\draw[red!80!black,thick,dotted,-stealth] ({1.25*cos(-60)},{1.25*sin(-60)}) arc (-60:60:1.2);

\draw[green!80!black,thick,dotted,-stealth] ({1.25*cos(120)},{1.25*sin(120)}) arc (120:240:1.2);

\draw[purple,thick,solid] (0.25,0) arc (0:140:0.25);
\node[black] at ({0.35*cos(65)},{0.35*sin(65)}) {bistatic angle $\alpha$};

\draw[black,thick,solid,fill=black] (0,0) circle (0.02cm);

\draw[cyan,fill=cyan] (-0.7,-0.2) circle (0.05cm) node[right,xshift=3pt,black] {$D_1$};

\draw[orange,fill=orange] (0.2,0.65) circle (0.035cm) node[above,black,yshift=2pt] {$D_2$};

\draw[violet,fill=violet] (0.5,-0.5) circle (0.04cm) node[above,black,yshift=2pt] {$D_3$};
\end{tikzpicture}
\caption{Illustration of bistatic measurement setup with bistatic angle $\alpha$}\label{setup}
\end{figure}
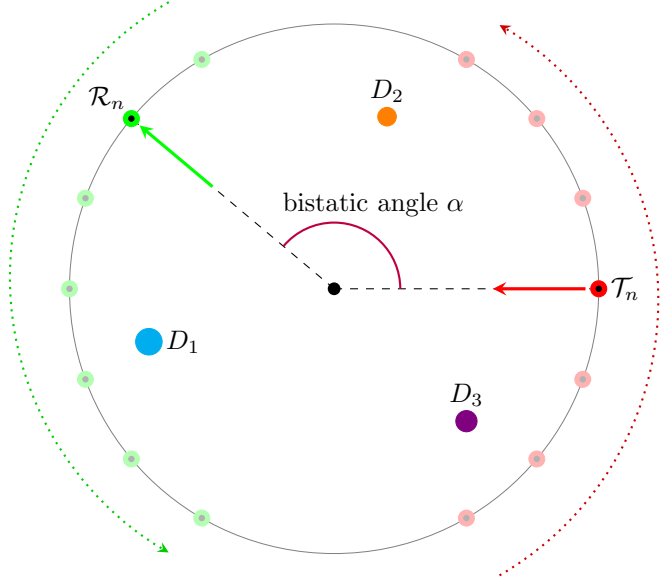

To introduce the bifocusing-based indicator function for identifying the locations $\mathbf{x}_m\in D_m$ from the set $\set{u_{\scat}(\mr_n,\mt_n):n=1,2,\ldots,N}$, let us generate an arrangement of the measurement data
\[\mF(\alpha)=\begin{bmatrix}
u_{\scat}(\mr_1,\mt_1)\\
u_{\scat}(\mr_2,\mt_2)\\
\vdots\\
u_{\scat}(\mr_N,\mt_N)
\end{bmatrix}.\]
Note that based on the mean-value theorem, there exists $\mz_m\in D_m$ such that
\[u_{\scat}(\mr_n,\mt_n)=k^2\sum_{m=1}^{M}\area(D_m)\left(\frac{\eps_m-\eps_0}{\eps_0\mu_0}\right)G(\mt_n,\mz_m)G(\mr_n,\mz_m),\]
the arrangement $\mF(\alpha)$ can be written as
\[\mF(\alpha)=k^2\sum_{m=1}^{M}\area(D_m)\left(\frac{\eps_m-\eps_0}{\eps_0\mu_0}\right)\begin{bmatrix}
G(\mt_1,\mz_m)G(\mr_1,\mz_m)\\
G(\mt_2,\mz_m)G(\mr_2,\mz_m)\\
\vdots\\
G(\mt_N,\mz_m)G(\mr_N,\mz_m)
\end{bmatrix}\]
and it leads us to introduce the following indicator function as follows: for each $\mx\in\Omega$, define
\[\mG(\mx)=\left[
\frac{1}{G(\mt_1,\mx)G(\mr_1,\mx)}\quad\frac{1}{G(\mt_2,\mx)G(\mr_2,\mx)}\quad\cdots\quad\frac{1}{G(\mt_N,\mx)G(\mr_N,\mx)}\right]^T.\]
Then, for given $\alpha$, the indicator function $\mathfrak{F}_{\bfm}(\mx,\alpha)$ can be introduced as
\begin{equation}\label{Indicator_Function}
\mathfrak{F}_{\bfm}(\mx,\alpha)=\abs{\mF(\alpha)\cdot\mG(\mx)}=\abs{\sum_{n=1}^{N}\frac{u_{\scat}(\mr_n,\mt_n)}{G(\mt_n,\mx)G(\mr_n,\mx)}}.
\end{equation}
Following the previous studies \cite{JBRBTFC,JTMCRBC,KLPS1,KLPS2,KP4,KPS,KJFF}, the map of $\mathfrak{F}_{\bfm}(\mx,\alpha)$ will exhibit peaks of relatively large magnitudes at $\mz_m\in D_m$ for all $m=1,2,\ldots,M$. Hence, it is expected that the existence and location of all $D_m$ can be retrieved through the map of $\mathfrak{F}_{\bfm}(\mx,\alpha)$.

\section{Theoretical result and various properties of the indicator function}\label{sec:3}
While the bifocusing-based indicator function guarantees the identification of inhomogeneities under a bistatic setup, this fact is insufficient to explain the applicability and the complex behavior observed in the numerical simulations in Section \ref{sec:4}. Given that the imaging performance relies heavily on the selection of the bistatic angle $\alpha$, we now derive the following result to analyze the intrinsic properties of $\mathfrak{F}_{\bfm}(\mx,\alpha)$.

\begin{theorem}[Structure of the indicator function]\label{Theorem_Function}
Assume that $4k|\mx-\mt_n|\gg1$ and $4k|\mx-\mr_n|\gg1$ for all $\mx\in\Omega$ and $n=1,2,\ldots,N$. Then, $\mathfrak{F}_{\bfm}(\mx,\alpha)$ can be represented as
\begin{multline}\label{Structure_Function}
\mathfrak{F}_{\bfm}(\mx,\alpha)=Nk^2\int_D\left(\frac{\eps(\mz)-\eps_0}{\eps_0\mu_0}\right)\bigg[J_0\big(k(1+\cos\alpha)|\mx-\mz|\big)J_0(k\sin\alpha|\mx-\mz|)\\
+2\sum_{q=1}^{\infty}(-1)^qJ_{2q}\big(k(1+\cos\alpha)|\mx-\mz|\big)J_{2q}(k\sin\alpha|\mx-\mz|)\bigg]d\mz,
\end{multline}
where $J_q$ denotes the Bessel function of order $q$ of the first kind.
\end{theorem}
\begin{proof}
Since $4k|\mx-\mt_n|\gg1$ and $4k|\mx-\mr_n|\gg1$ for all $n=1,2,\ldots,N$, the following asymptotic form of the Hankel function holds (see \cite{CK} for instance):
\begin{equation}\label{Asymptotic_Hankel}
G(\mt_n,\mx)=-\frac{(1+i)e^{ikT}}{4\sqrt{kT\pi}}e^{-ik\vt_n\cdot\mx}\qand G(\mr_n,\mx)=-\frac{(1+i)e^{ikR}}{4\sqrt{kR\pi}}e^{-ik\vv_n\cdot\mx}.
\end{equation}
Then, $\mF(\alpha)$ and $\mG(\mx)$ are written as
\[\mF(\alpha)=k^2\int_D\left(\frac{\eps(\mz)-\eps_0}{\eps_0\mu_0}\right)\frac{ie^{ik(T+R)}}{8k\pi\sqrt{TR}}\begin{bmatrix}
e^{-ik(\vt_1+\vv_1)\cdot\mz}\\
e^{-ik(\vt_2+\vv_2)\cdot\mz}\\
\vdots\\
e^{-ik(\vt_N+\vv_N)\cdot\mz}
\end{bmatrix}d\mz\]
and
\[\mG(\mx)=\frac{8k\pi\sqrt{TR}}{ie^{ik(T+R)}}\begin{bmatrix}
e^{ik(\vt_1+\vv_1)\cdot\mx}\\
e^{ik(\vt_2+\vv_2)\cdot\mx}\\
\vdots\\
e^{ik(\vt_N+\vv_N)\cdot\mx}
\end{bmatrix},\]
respectively. Then,
\[\mF(\alpha)\cdot\mG(\mx)=k^2\sum_{n=1}^{N}\int_D\left(\frac{\eps(\mz)-\eps_0}{\eps_0\mu_0}\right)e^{ik(\vt_n+\vv_n)\cdot(\mx-\mz)}d\mz.\]

Since 
\begin{align*}
&(\cos\theta_n+\cos(\theta_n+\alpha),\sin\theta_n+\sin(\theta_n+\alpha))\cdot(\cos\phi_m,\sin\phi_m)\\
&=(\cos\theta_n+\cos\theta_n\cos\alpha-\sin\theta_n\sin\alpha,\sin\theta_n+\sin\theta_n\cos\alpha+\cos\theta_n\sin\alpha)\cdot(\cos\phi_m,\sin\phi_m)\\
&=(1+\cos\alpha)\cos\theta_n\cos\phi_m-\sin\alpha\sin\theta_n\cos\xi+(1+\cos\alpha)\sin\theta_n\sin\phi_m+\sin\alpha\cos\theta_n\sin\phi_m\\
&=(1+\cos\alpha)\cos(\theta_n-\phi_m)-\sin\alpha\sin(\theta_n-\phi_m),
\end{align*}
by letting $\mx-\mz=|\mx-\mz|(\cos\phi,\sin\phi)$, we can write
\[(\vt_n+\vv_n)\cdot(\mx-\mz)=|\mx-\mz|\big((1+\cos\alpha)\cos(\theta_n-\phi)-\sin\alpha\sin(\theta_n-\phi)\big)\]
and $\mF(\alpha)\cdot\mG(\mx)$ becomes
\begin{align*}
\mF(\alpha)\cdot\mG(\mx)&=k^2\int_D\left(\frac{\eps(\mz)-\eps_0}{\eps_0\mu_0}\right)\sum_{n=1}^{N}e^{ik|\mx-\mz|\big((1+\cos\alpha)\cos(\theta_n-\phi)-\sin\alpha\sin(\theta_n-\phi)\big)}d\mz\\
&=k^2\int_D\left(\frac{\eps(\mz)-\eps_0}{\eps_0\mu_0}\right)\frac{N}{N\triangle\theta}\sum_{n=1}^{N}e^{ik|\mx-\mz|\big((1+\cos\alpha)\cos(\theta_n-\phi)-\sin\alpha\sin(\theta_n-\phi)\big)}\triangle\theta d\mz\\
&\approx \frac{Nk^2}{2\pi}\int_D\left(\frac{\eps(\mz)-\eps_0}{\eps_0\mu_0}\right)\left(\int_0^{2\pi}e^{ik|\mx-\mz|\big((1+\cos\alpha)\cos(\theta-\phi)-\sin\alpha\sin(\theta-\phi)\big)}d\theta\right)d\mz.
\end{align*}

Since the following Jacobi-Anger expansion formulas hold uniformly:
\begin{align*}
e^{ix\cos\theta}&=J_0(x)+2\sum_{p=1}^{\infty}i^{p}J_{p}(x)\cos(p\theta),\\
e^{-ix\sin\theta}&=J_0(x)+2\sum_{q=1}^{\infty}J_{2q}(x)\cos(2q\theta)-2i\sum_{q=0}^{\infty}J_{2q+1}(x)\sin\big((2q+1)\theta\big)
\end{align*}
we can write
\begin{align}
\begin{aligned}\label{Term}
&\int_0^{2\pi}e^{ik|\mx-\mz|\big((1+\cos\alpha)\cos(\theta-\phi)-\sin\alpha\sin(\theta-\phi)\big)}d\theta\\
&=\int_0^{2\pi}\big(\mathcal{I}_1+\mathcal{I}_2+\mathcal{I}_3+\mathcal{I}_4+\mathcal{I}_5+\mathcal{I}_6\big)d\theta,
\end{aligned}
\end{align}
where
\begin{align*}
\mathcal{I}_1&=J_0(k(1+\cos\alpha)|\mx-\mz|)J_0(k\sin\alpha|\mx-\mz|)\\
\mathcal{I}_2&=2J_0(k(1+\cos\alpha)|\mx-\mz|)\sum_{q=1}^{\infty}J_{2q}(k\sin\alpha|\mx-\mz|)\cos\big(2q(\theta-\phi)\big)\\
\mathcal{I}_3&=-2iJ_0(k(1+\cos\alpha)|\mx-\mz|)\sum_{q=0}^{\infty}J_{2q+1}(k\sin\alpha|\mx-\mz|)\sin\big((2q+1)(\theta-\phi)\big)\\
\mathcal{I}_4&=2J_0(k\sin\alpha|\mx-\mz|)\sum_{p=1}^{\infty}i^{p}J_{p}(k(1+\cos\alpha)|\mx-\mz|)\cos\big(p(\theta-\phi)\big)\\
\mathcal{I}_5&=4\left(\sum_{p=1}^{\infty}i^{p}J_{p}(k(1+\cos\alpha)|\mx-\mz|)\cos\big(p(\theta-\phi)\big)\right)\\
&\phantom{=4}\times\left(\sum_{q=1}^{\infty}J_{2q}(k\sin\alpha|\mx-\mz|)\cos\big(2q(\theta-\phi)\big)\right)\\
\mathcal{I}_6&=-4i\left(\sum_{p=1}^{\infty}i^{p}J_{p}(k(1+\cos\alpha)|\mx-\mz|)\cos\big(p(\theta-\phi)\big)\right)\\
&\phantom{=-4i}\times\left(\sum_{q=0}^{\infty}J_{2q+1}(k\sin\alpha|\mx-\mz|)\sin\big((2q+1)(\theta-\phi)\big)\right)
\end{align*}
Let $k(1+\cos\alpha)|\mx-\mz|=\varphi_1$ and $k\sin\alpha|\mx-\mz|=\varphi_2$. By performing a direct calculation, we can immediately evaluate
\begin{align}
\begin{aligned}\label{Term1234}
\int_0^{2\pi}\mathcal{I}_1d\theta&=\int_0^{2\pi}J_0(\varphi_1)J_0(\varphi_2)d\theta=2\pi J_0\big(k(1+\cos\alpha)|\mx-\mz|\big)J_0(k\sin\alpha|\mx-\mz|)\\
\int_0^{2\pi}\mathcal{I}_2d\theta&=2J_0(\varphi_1)\sum_{q=1}^{\infty}J_{2q}(\varphi_2)\int_0^{2\pi}\cos\big(2q(\theta-\phi)\big)d\theta=0\\
\int_0^{2\pi}\mathcal{I}_3d\theta&=-2iJ_0(\varphi_1)\sum_{q=0}^{\infty}J_{2q+1}(\varphi_2)\int_0^{2\pi}\sin\big((2q+1)(\theta-\phi)\big)d\theta=0\\
\int_0^{2\pi}\mathcal{I}_4d\theta&=2J_0(\varphi_2)\sum_{p=1}^{\infty}i^{p}J_{p}(\varphi_1)\int_0^{2\pi}\cos\big(p(\theta-\phi)\big)d\theta=0.
\end{aligned}
\end{align}

For any integers $p$ and $q$, since
\[\int_0^{2\pi}\cos\big(p(\theta-\phi_s)\big)\sin\big(q(\theta-\phi_s)\big)d\theta=0\]
and
\[\int_0^{2\pi}\cos\big(p(\theta-\phi_s)\big)\cos\big(q(\theta-\phi_s)\big)d\theta=\left\{\begin{array}{cl}
\smallskip\pi,&p=q,\\
0,&p\ne q,
\end{array}\right.\]
we have
\begin{align}
\begin{aligned}\label{Term5}
\int_0^{2\pi}\mathcal{I}_5d\theta&=4\sum_{p=1}^{\infty}\sum_{q=1}^{\infty}i^{p}J_{p}(\varphi_1)J_{2q}(\varphi_2)\int_0^{2\pi}\cos\big(p(\theta-\phi)\big)\cos\big(2q(\theta-\phi)\big)d\theta\\
&=4\pi\sum_{q=1}^{\infty}(-1)^qJ_{2q}\big(k(1+\cos\alpha)|\mx-\mz|\big)J_{2q}(k\sin\alpha|\mx-\mz|)
\end{aligned}
\end{align}
and
\begin{equation}\label{Term6}
\int_0^{2\pi}\mathcal{I}_6d\theta=-4i\sum_{p=1}^{\infty}\sum_{q=0}^{\infty}i^{p}J_{p}(\varphi_1)J_{2q+1}(\varphi_2)\int_0^{2\pi}\cos\big(p(\theta-\phi)\big)\sin\big((2q+1)(\theta-\phi)\big)d\theta=0.
\end{equation}

Finally, by plugging \eqref{Term1234}, \eqref{Term5}, and \eqref{Term6} into \eqref{Term},
\begin{align*}
\int_0^{2\pi}&e^{ik|\mx-\mz|\big((1+\cos\alpha)\cos(\theta-\phi)-\sin\alpha\sin(\theta-\phi)\big)}d\theta\\
&=2\pi J_0\big(k(1+\cos\alpha)|\mx-\mz|\big)J_0(k\sin\alpha|\mx-\mz|)\\
&+4\pi\sum_{q=1}^{\infty}(-1)^qJ_{2q}\big(k(1+\cos\alpha)|\mx-\mz|\big)J_{2q}(k\sin\alpha|\mx-\mz|)
\end{align*}
and \eqref{Structure_Function} can be obtained immediately.
\end{proof}

\begin{corollary}\label{Corollary_Function}
For $0\leq\alpha\leq\pi$, the following relationship holds
\[\mathfrak{F}_{\bfm}(\mx,2\pi-\alpha)=\mathfrak{F}_{\bfm}(\mx,\alpha).\]
\end{corollary}
\begin{proof}
Since $\cos(2\pi-\alpha)=\cos\alpha$, $\sin(2\pi-\alpha)=-\sin\alpha$, and $J_{2p}(-x)=(-1)^{2p}J_{2p}(x)=J_{2p}(x)$ for any integer $p$,
\begin{align*}
J_{2q}&\big(k(1+\cos(2\pi-\alpha))|\mx-\mz|\big)J_{2q}(k\sin(2\pi-\alpha)|\mx-\mz|)\\
&=J_{2q}\big(k(1+\cos\alpha)|\mx-\mz|\big)J_{2q}(k\sin\alpha|\mx-\mz|)
\end{align*}
for $q=0,1,\ldots$. Therefore, we can examine that $\mathfrak{F}_{\bfm}(\mx,2\pi-\alpha)=\mathfrak{F}_{\bfm}(\mx,\alpha)$.
\end{proof}

Based on Theorem \ref{Theorem_Function} and Corollary \ref{Corollary_Function}, we can discover several properties of the indicator function. They can be summarized as follows.

\begin{remark}[Applicability of the detection]\label{Remark1}
First, we consider the following term in \eqref{Structure_Function}
\[\Phi_1(\mx,\alpha)=J_0\big(k(1+\cos\alpha)|\mx-\mz|\big)J_0(k\sin\alpha|\mx-\mz|).\]
Since $J_0(0)=1$, $\mathfrak{F}_{\bfm}(\mx,\alpha)$ will have its maximum value when $\mx=\mz\in D$, i.e., the term $\Phi_1(\mx,\alpha)$ will contribute to identify the existence and location of $D_m$ unless $\alpha\ne\pi$. This means that by regarding peaks of relatively large magnitudes in the map of $\mathfrak{F}_{\bfm}(\mx,\alpha)$, every inhomogeneity $D_m$ can be recognized. Now, let us consider the following:
\[\mathcal{E}_1(x,\alpha)=|J_0\big(k(1+\cos\alpha)|x|\big)J_0(k\sin\alpha|x|)|.\]
This is the 1D version of $J_0\big(k(1+\cos\alpha)|\mx-\mz|\big)J_0(k\sin\alpha|\mx-\mz|)$ when $\mz$ is the origin. By examining the plots of $\mathcal{E}_1(x,\alpha)$ with various $\alpha$ in Fig. \ref{PlotE1}, we can observe that the resolution of the imaging results degrades as $\alpha$ approaches $\pi$, whereas relatively high-resolution results are obtained when $\alpha$ is close to $0$.
\end{remark}

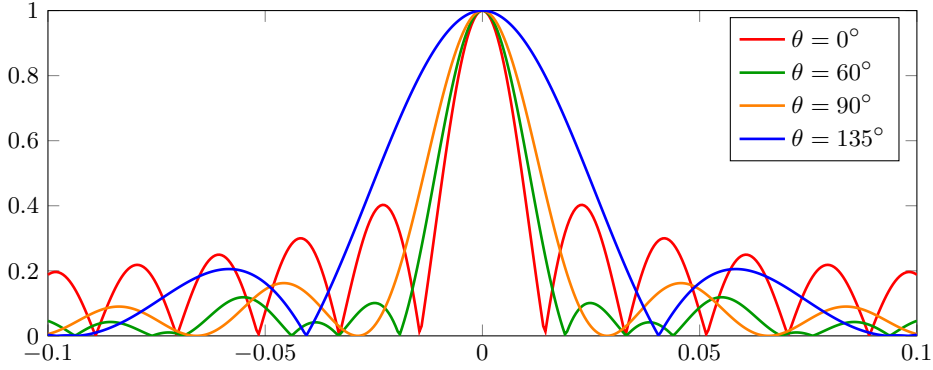
\begin{figure}[h]
\centering
\begin{tikzpicture}
\scriptsize
\begin{axis}
[width=\textwidth,
height=0.45\textwidth,
enlarge x limits = false,
enlarge y limits = false,
xtick={-0.1,-0.05,0,0.05,0.1},
xticklabel style={/pgf/number format/fixed},
legend cell align={left}]
\addplot[line width=1pt,solid,color=red] %
	table[x=x,y=y1,col sep=comma]{PlotE1.csv};
\addlegendentry{\scriptsize$\theta=\SI{0}{\degree}$}

\addplot[line width=1pt,solid,color=green!60!black] %
	table[x=x,y=y2,col sep=comma]{PlotE1.csv};
\addlegendentry{\scriptsize$\theta=\SI{60}{\degree}$}

\addplot[line width=1pt,solid,color=orange] %
	table[x=x,y=y3,col sep=comma]{PlotE1.csv};
\addlegendentry{\scriptsize$\theta=\SI{90}{\degree}$}

\addplot[line width=1pt,solid,color=blue] %
	table[x=x,y=y4,col sep=comma]{PlotE1.csv};
\addlegendentry{\scriptsize$\theta=\SI{135}{\degree}$}
\end{axis}
\end{tikzpicture}
\caption{\label{PlotE1}Plots of $\mathcal{E}_1(x,\alpha)$ with $f=\SI{4}{\GHz}$ and various $\alpha$.}
\end{figure}

\begin{remark}[Appearance of some artifacts]\label{Remark2}
Now, let us consider the following term in \eqref{Structure_Function}
\[\Phi_2(\mx,\alpha)=2\sum_{q=1}^{\infty}(-1)^qJ_{2q}\big(k(1+\cos\alpha)|\mx-\mz|\big)J_{2q}(k\sin\alpha|\mx-\mz|).\]
Since $J_q(0)=0$ for any nonzero $q$, $\Phi_2(\mx,\alpha)$ does not contribute to identify $D_m$. Moreover, based on the oscillation pattern of the function $J_q$, some artifacts whose magnitudes are relatively smaller than the ones on $D_m$ are also included in the map of $\mathfrak{F}_{\bfm}(\mx,\alpha)$. For a proper discussion, we consider the following:
\[\mathcal{E}_2(x,\alpha)=\abs{2\sum_{q=1}^{10^5}(-1)^qJ_{2q}\big(k(1+\cos\alpha)|x|\big)J_{2q}(k\sin\alpha|x|)}.\]
By regarding the plots of $\mathcal{E}_2(x,\alpha)$ with various $\alpha$ in Fig. \ref{PlotE2}, we can observe that $\Phi_2(\mx,\alpha)$ will disturb the recognition of $D_m$ and contribute to produce several artifacts.
\end{remark}

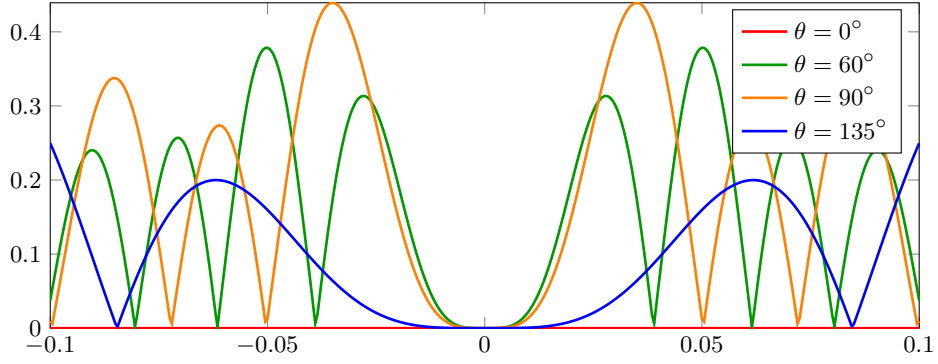
\begin{figure}[h]
\centering
\begin{tikzpicture}
\scriptsize
\begin{axis}
[width=\textwidth,
height=0.45\textwidth,
enlarge x limits = false,
enlarge y limits = false,
xtick={-0.1,-0.05,0,0.05,0.1},
xticklabel style={/pgf/number format/fixed},
legend cell align={left}]
\addplot[line width=1pt,solid,color=red] %
	table[x=x,y=y1,col sep=comma]{PlotE2.csv};
\addlegendentry{\scriptsize$\theta=\SI{0}{\degree}$}

\addplot[line width=1pt,solid,color=green!60!black] %
	table[x=x,y=y2,col sep=comma]{PlotE2.csv};
\addlegendentry{\scriptsize$\theta=\SI{60}{\degree}$}

\addplot[line width=1pt,solid,color=orange] %
	table[x=x,y=y3,col sep=comma]{PlotE2.csv};
\addlegendentry{\scriptsize$\theta=\SI{90}{\degree}$}

\addplot[line width=1pt,solid,color=blue] %
	table[x=x,y=y4,col sep=comma]{PlotE2.csv};
\addlegendentry{\scriptsize$\theta=\SI{135}{\degree}$}
\end{axis}
\end{tikzpicture}
\caption{\label{PlotE2}Plots of $\mathcal{E}_2(x,\alpha)$ with $f=\SI{4}{\GHz}$ and various $\alpha$.}
\end{figure}

\begin{remark}[Specific cases: $\alpha=0$ and $\alpha=\pi$]\label{Remark3}
Assume that $\alpha=0$ then this problem becomes an identification of inhomogeneities in monostatic measurement setup. Then since $1+\cos\alpha=2$, $\sin\alpha=0$, $J_0(0)=1$, and $J_q(0)=0$ for nonzero $q$, $\mathfrak{F}_{\bfm}(\mx,\alpha)$ becomes
\[\mathfrak{F}_{\bfm}(\mx,0)=Nk^2\int_D\left(\frac{\eps(\mz)-\eps_0}{\eps_0\mu_0}\right)J_0(2k|\mx-\mz|)d\mz.\]
Hence, it will be possible to identify the existence and location of $D_m$. It is noteworthy that this result aligns closely with the theoretical results derived in the study of the direct sampling method \cite{KLP3}.

In contrast, if $\alpha=\pi$ then since $1+\cos\alpha=0$ and $\sin\alpha=0$, $\mathfrak{F}_{\bfm}(\mx,\alpha)$ becomes
\[\mathfrak{F}_{\bfm}(\mx,\pi)=Nk^2\int_D\left(\frac{\eps(\mz)-\eps_0}{\eps_0\mu_0}\right)d\mz.\]
This means that the value of $\mathfrak{F}_{\bfm}(\mx,\pi)$ does not depend on the existence of $D_m$. Hence, it will be impossible to identify the existence and location of $D_m$.
\end{remark}

\begin{remark}[Unique identification of inhomogeneities]\label{Remark4}
Based on Remarks \ref{Remark1} and \ref{Remark2}, we can observe that the term $\Phi_1(\mx,\alpha)$ reaches its maximum value $1$ when $\mx\in D_m$ and based on numerical computation, the term $\Phi_2(\mx,\alpha)$ reaches its maximum value $0.4392$ when $|\mx-\mz|=\pm0.035$ (see Fig. \ref{PlotE2} also). Hence, in the neighborhood of $D_m$ the value of $\Phi_2(\mx,\alpha)$ will be dominated by $\Phi_1(\mx,\alpha)$ thus, the term $\Phi_2(\mx,\alpha)$ will not significantly compromise the detection of $D_m$. We also refer to Fig. \ref{PlotE3} for an illustration of
\[\mathcal{E}(x,\alpha)=\abs{J_0\big(k(1+\cos\alpha)|x|\big)J_0(k\sin\alpha|x|)+2\sum_{q=1}^{10^5}(-1)^qJ_{2q}\big(k(1+\cos\alpha)|x|\big)J_{2q}(k\sin\alpha|x|)}.\]
\end{remark}

\begin{figure}[h]
\centering
\begin{tikzpicture}
\scriptsize
\begin{axis}
[width=\textwidth,
height=0.45\textwidth,
enlarge x limits = false,
enlarge y limits = false,
xtick={-0.1,-0.05,0,0.05,0.1},
xticklabel style={/pgf/number format/fixed},
legend cell align={left}]
\addplot[line width=1pt,solid,color=red] %
	table[x=x,y=y1,col sep=comma]{PlotE3.csv};
\addlegendentry{\scriptsize$\theta=\SI{0}{\degree}$}

\addplot[line width=1pt,solid,color=green!60!black] %
	table[x=x,y=y2,col sep=comma]{PlotE3.csv};
\addlegendentry{\scriptsize$\theta=\SI{60}{\degree}$}

\addplot[line width=1pt,solid,color=orange] %
	table[x=x,y=y3,col sep=comma]{PlotE3.csv};
\addlegendentry{\scriptsize$\theta=\SI{90}{\degree}$}

\addplot[line width=1pt,solid,color=blue] %
	table[x=x,y=y4,col sep=comma]{PlotE3.csv};
\addlegendentry{\scriptsize$\theta=\SI{135}{\degree}$}
\end{axis}
\end{tikzpicture}
\caption{\label{PlotE3}Plots of $\mathcal{E}(x,\alpha)$ with $f=\SI{4}{\GHz}$ and various $\alpha$.}
\end{figure}
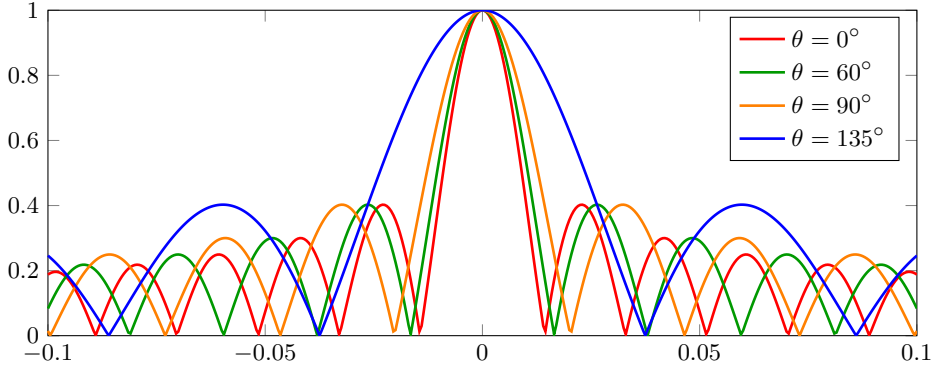

\section{Numerical simulation results}\label{sec:4}
In this section, we exhibit numerical simulation results to support the theoretical results and elucidate the discovered properties of the indicator function $\mathfrak{F}_{\bfm}(\mx,\alpha)$. The measured scattered field data was extracted from the two-dimensional Fresnel experimental dataset \cite{BS}. We considered the cases of a single dielectric circular inhomogeneity (\texttt{dielTM\_dec8f.exp}), two dielectric circular inhomogeneities (\texttt{twodielTM\_dec8f.exp}), and a single metallic rectangle (\texttt{rectTM\_dece.exp}). In addition, a $\SI{20}{\dB}$ white Gaussian random noise was added to the unperturbed scattered field data and $N=36$ number of measurement data with step size $\triangle\theta=\SI{10}{\degree}$ were used to generated imaging results. The imaging region $\Omega$ was selected as the square domain $[-\SI{0.1}{\meter},\SI{0.1}{\meter}]\times[-\SI{0.1}{\meter},\SI{0.1}{\meter}]$. Note that instead of the exhibition of the imaging result $\mathfrak{F}_{\bfm}(\mx,\alpha)$, we show the imaging results of the following normalized indicator function:
\[\mathfrak{F}_{\bfn}(\mx,\alpha)=\frac{\mathfrak{F}_{\bfm}(\mx,\alpha)}{\displaystyle\max_{\mx\in\Omega}\left\{\mathfrak{F}_{\bfm}(\mx,\alpha)\right\}}.\]

\begin{example}[Imaging of single dielectric circular inhomogeneity]\label{Ex1}
Fig. \ref{Fig1} shows the maps of $\mathfrak{F}_{\bfn}(\mx,\alpha)$ at a frequency of operation $f=\SI{4}{\GHz}$ with various bistatic angles $\alpha$ in the presence of a single dielectric circular inhomogeneity $D$ with permittivity $(3\pm0.3)\eps_0$ centered at $(-\SI{0.030}{\meter},\SI{0.000}{\meter})$. Based on the imaging results, we can examine that for $\alpha=\SI{60}{\degree}$ and $\SI{90}{\degree}$, the existence and outline shape of $D$ is clearly identified with high resolution. However, as $\alpha$ increases to $\SI{120}{\degree}$, $\SI{135}{\degree}$, and $\SI{150}{\degree}$, the imaging results become blurred, and the resolution gradually degrades. This observation aligns with the property discussed in Remark \ref{Remark1}, which states that resolution degrades as $\alpha$ approaches $\SI{180}{\degree}$, whereas relatively high-resolution results are obtained when $\alpha$ is close to $\SI{0}{\degree}$. Note that when $\alpha=\SI{180}{\degree}$, the existence and location of the $D$ cannot be recognized through the map of $\mathfrak{F}_{\bfn}(\mx,\alpha)$. This validates the theoretical result in Theorem \ref{Theorem_Function} and discussion in Remark \ref{Remark3}, which asserts that the value of the indicator function does not depend on the existence of inhomogeneities when $\alpha=\SI{180}{\degree}$.
\end{example}

\begin{figure}[h]
\includegraphics[width=.33\textwidth]{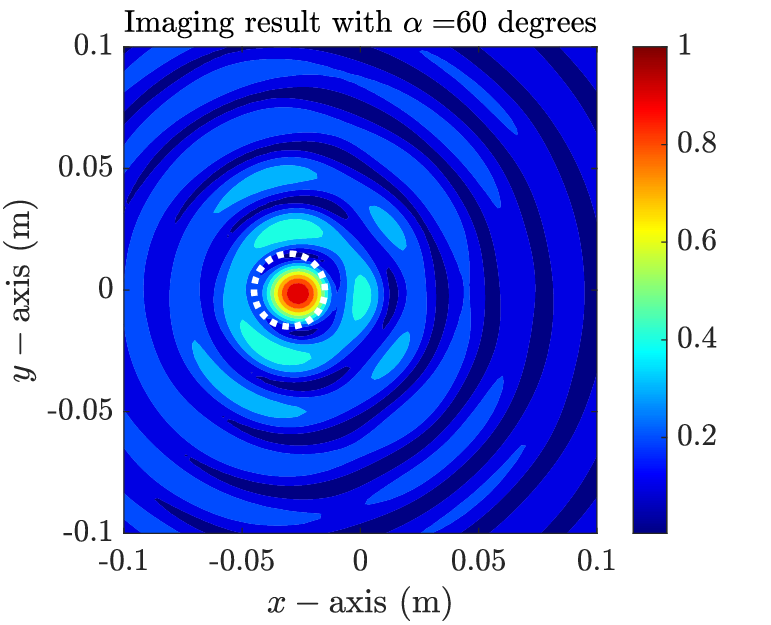}\hfill
\includegraphics[width=.33\textwidth]{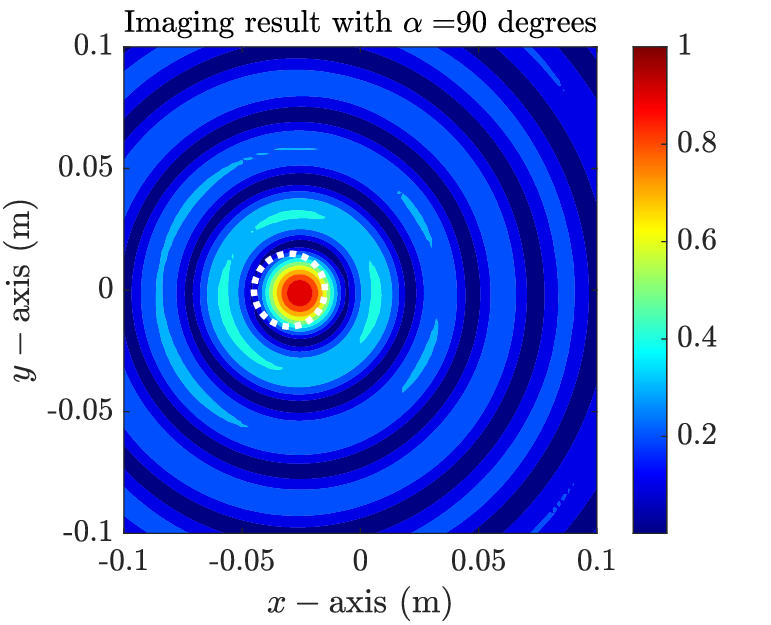}\hfill
\includegraphics[width=.33\textwidth]{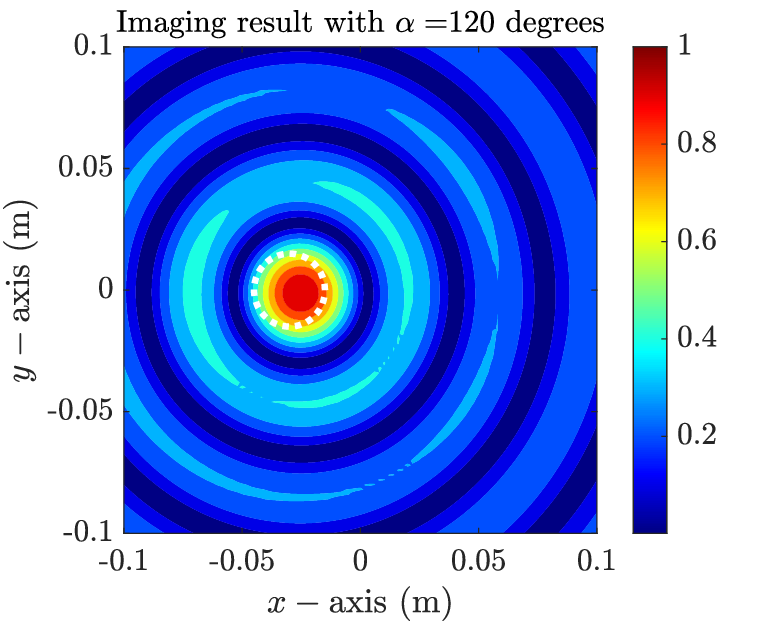}\\
\includegraphics[width=.33\textwidth]{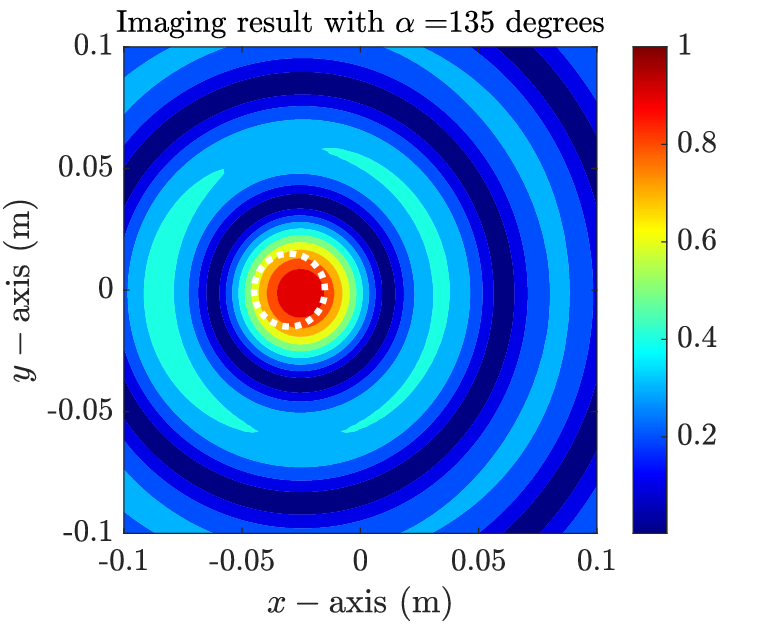}\hfill
\includegraphics[width=.33\textwidth]{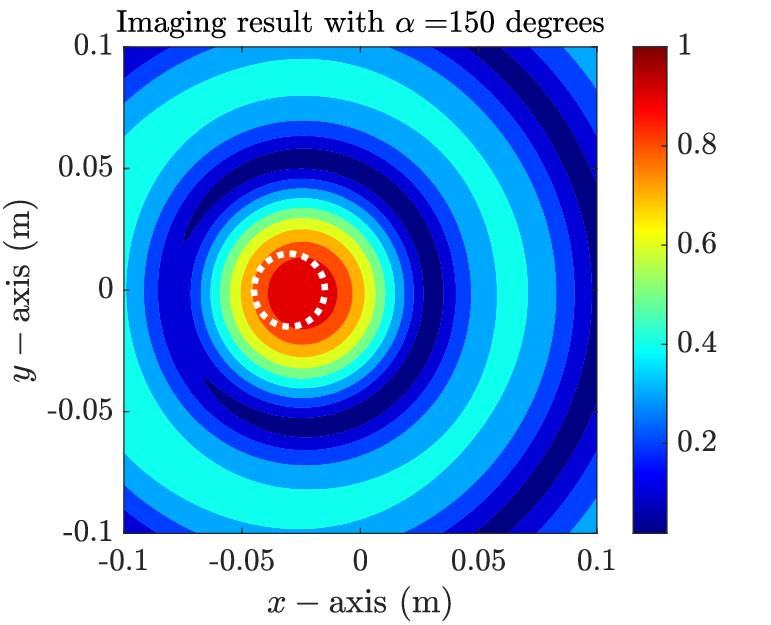}\hfill
\includegraphics[width=.33\textwidth]{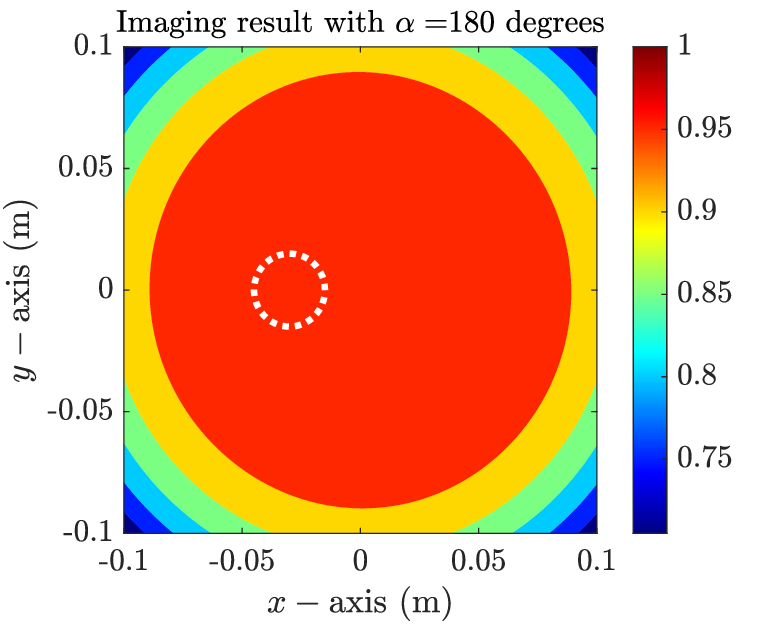}
\caption{\label{Fig1}(Example \ref{Ex1}) Maps of $\mathfrak{F}_{\bfn}(\mx,\alpha)$ at $f=\SI{4}{\GHz}$ with various $\alpha$.}
\end{figure}

\begin{example}[Imaging of two dielectric circular inhomogeneities]\label{Ex2}
Fig. \ref{Fig2} displays the imaging results in the presence of two dielectric circular inhomogeneities $D_1$ and $D_2$ centered at $(-\SI{0.045}{\meter},\SI{0.000}{\meter})$ and $(\SI{0.045}{\meter},\SI{0.010}{\meter})$, respectively, and with same permittivity $(3\pm0.3)\eps_0$. In contrast to the theoretical result and imaging of single inhomogeneity, it is very difficult to identify $D_1$ and $D_2$ when $\alpha=\SI{60}{\degree}$ and $\alpha=\SI{150}{\degree}$. Note that although peaks of large magnitudes were appeared when $\mx\in D_1\cup D_2$, due to the appearance of another peaks of large magnitudes, the distinction between the two objects becomes ambiguous when $\alpha=\SI{90}{\degree}$ and $\alpha=\SI{120}{\degree}$. Fortunately, when $\alpha=\SI{135}{\degree}$, $D_1$ and $D_2$ are clearly identified through the imaging result. However, similar to the single object case, identification becomes impossible when $\alpha$ reaches $\SI{180}{\degree}$. Hence, further improvement of the imaging performance is required still.
\end{example}

\begin{figure}[h]
\includegraphics[width=.33\textwidth]{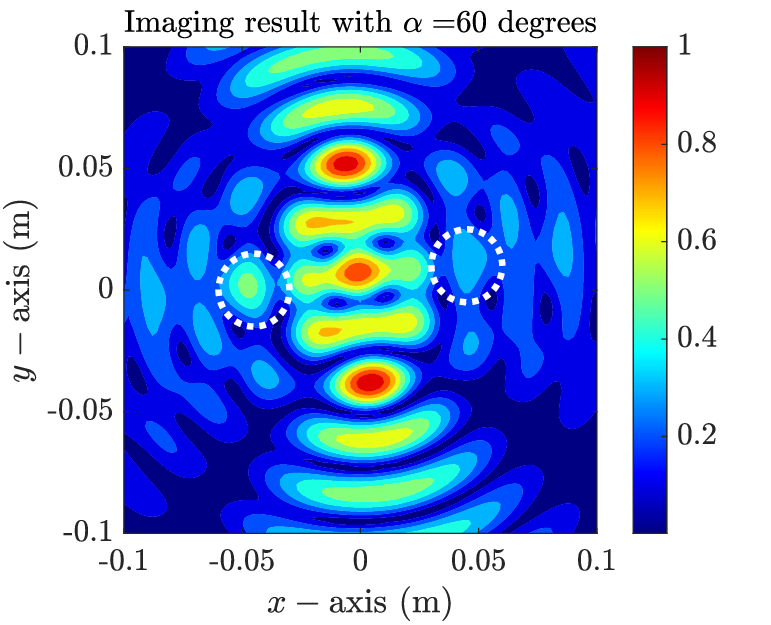}\hfill
\includegraphics[width=.33\textwidth]{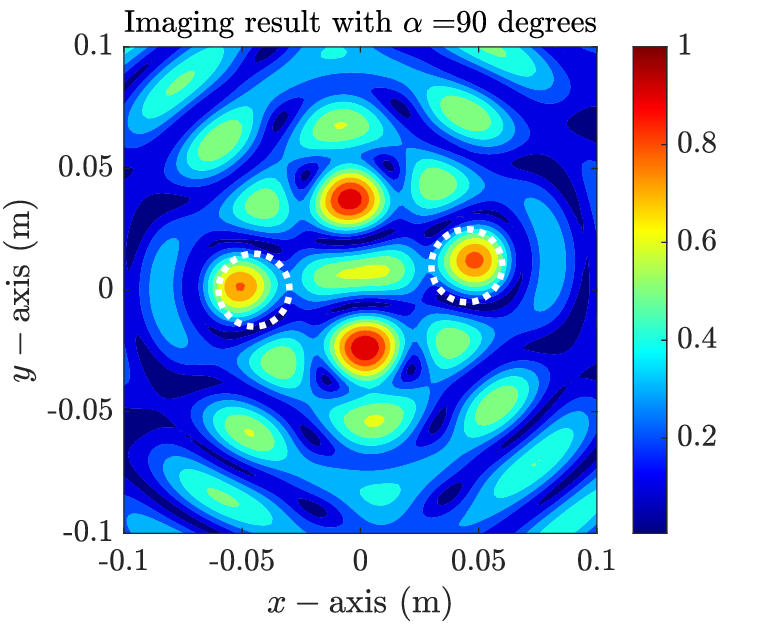}\hfill
\includegraphics[width=.33\textwidth]{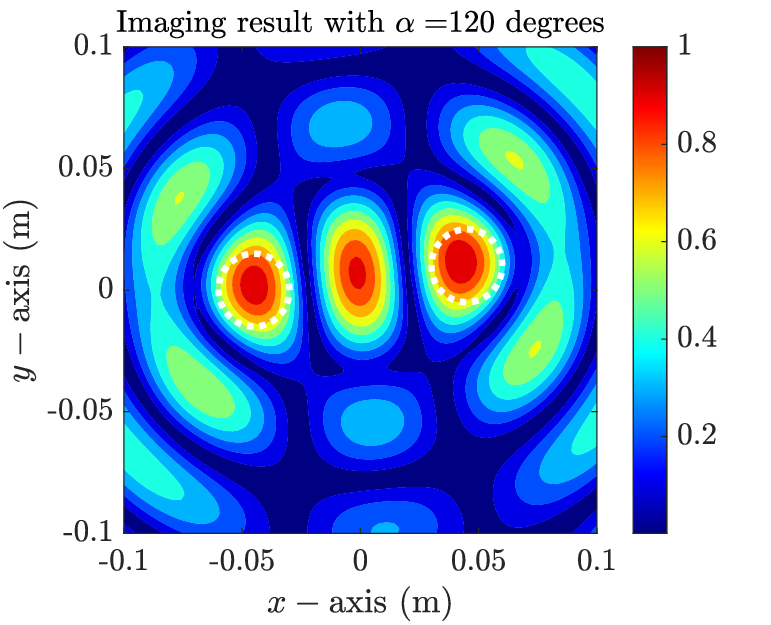}\\
\includegraphics[width=.33\textwidth]{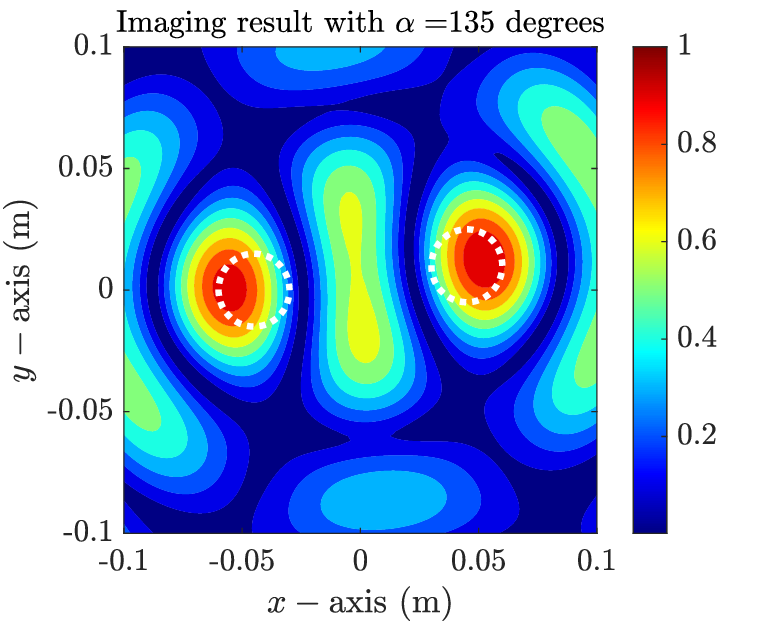}\hfill
\includegraphics[width=.33\textwidth]{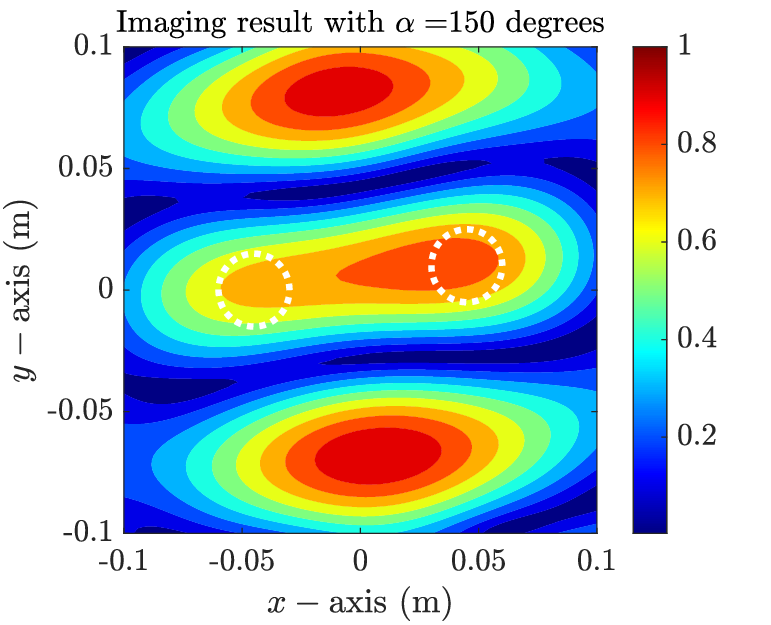}\hfill
\includegraphics[width=.33\textwidth]{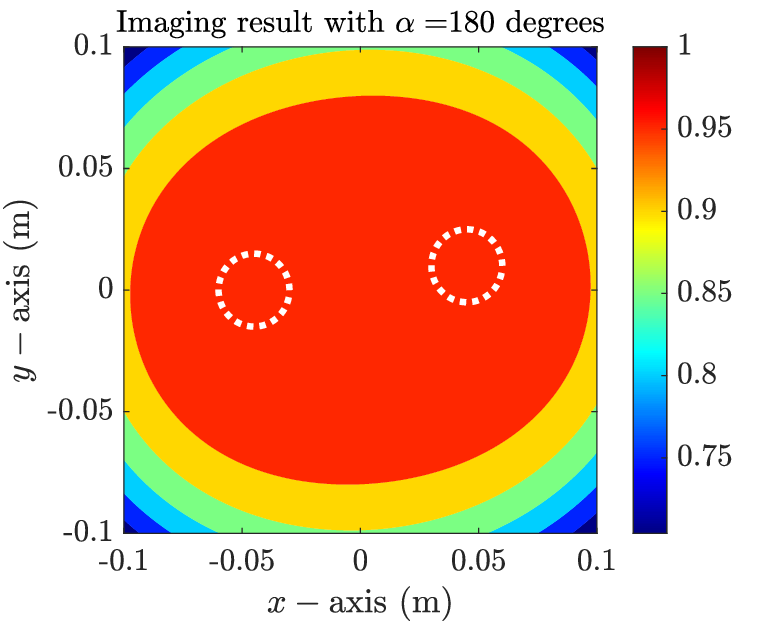}
\caption{\label{Fig2}(Example \ref{Ex2}) Maps of $\mathfrak{F}_{\bfn}(\mx,\alpha)$ at $f=\SI{4}{\GHz}$ with various $\alpha$.}
\end{figure}

\begin{example}[Imaging of single metallic rectangle]\label{Ex3}
Now, we apply the indicator function for identifying a perfectly conducting material with non-smooth boundary. Fig. \ref{Fig3} illustrates the imaging results for a single metallic rectangular object $D$ with dimension $\SI{0.0254}{\m}\times\SI{0.0127}{\m}$. Although the theoretical background of this study assumes dielectric inhomogeneities, the results demonstrate that the designed indicator function operates effectively for metallic object with non-smooth boundary as well. Similar to the Example \ref{Ex1}, the location of $D$ was relatively well-reconstructed when $\alpha=\SI{60}{\degree}$ and $\SI{90}{\degree}$. However, as $\alpha$ increases, the image spreads significantly, and imaging fails at $\alpha=\SI{180}{\degree}$. This suggests that the bistatic angle $\alpha$ is a critical factor determining imaging performance. It is noteworthy to emphasize that although a reliable result has been obtained, exact shape of rectangular object cannot be retrieved. 
\end{example}

\begin{figure}[h]
\includegraphics[width=.33\textwidth]{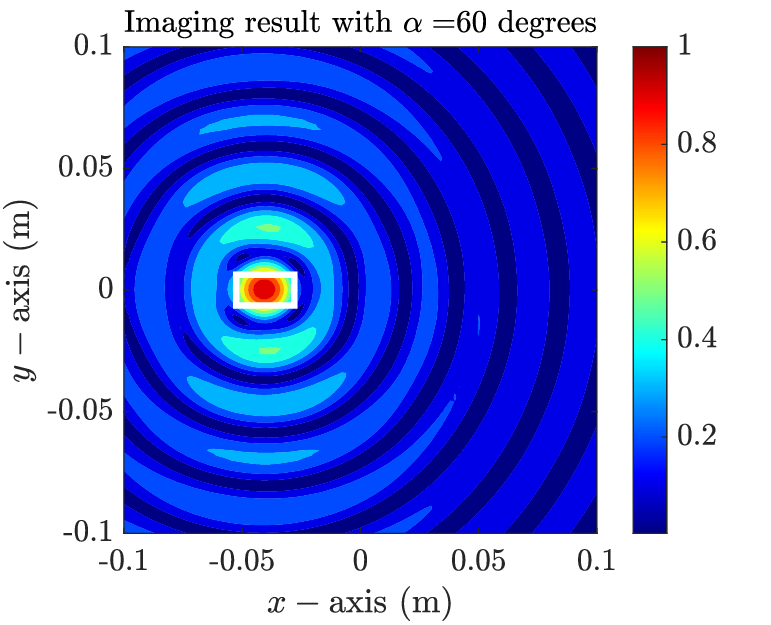}\hfill
\includegraphics[width=.33\textwidth]{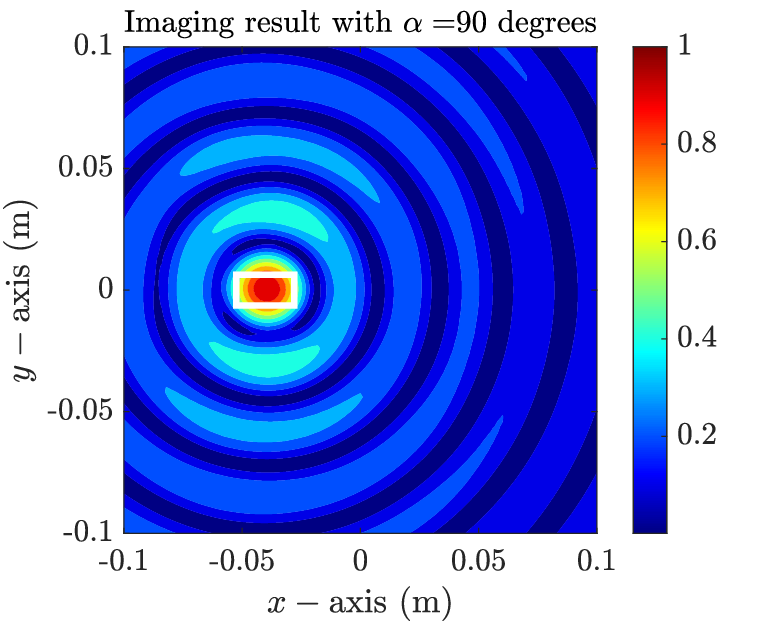}\hfill
\includegraphics[width=.33\textwidth]{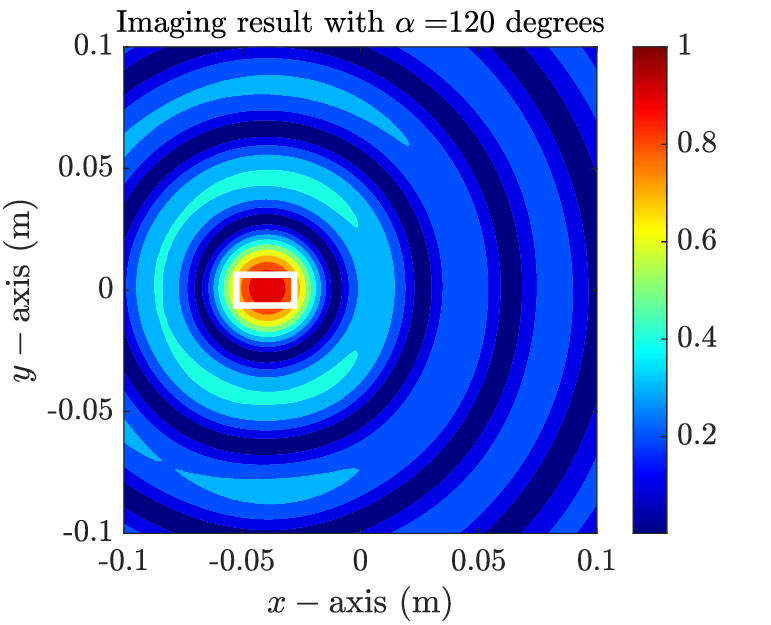}\\
\includegraphics[width=.33\textwidth]{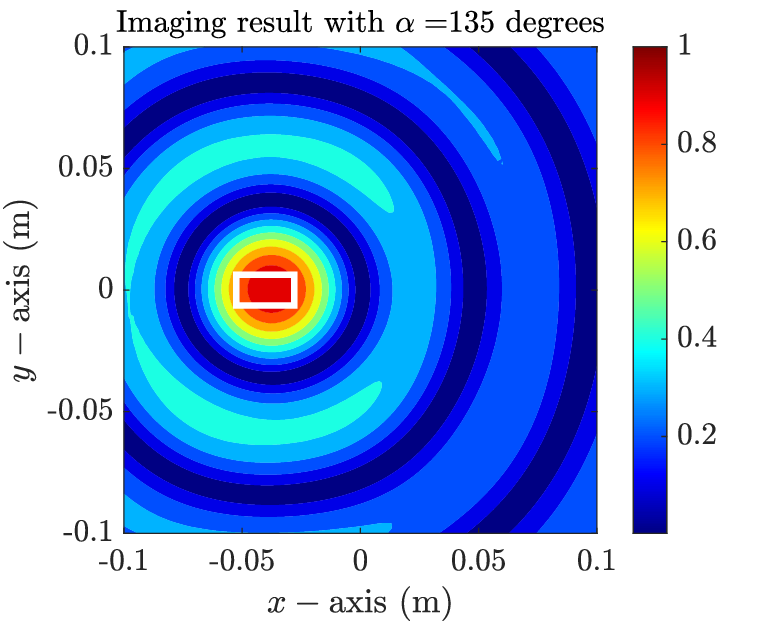}\hfill
\includegraphics[width=.33\textwidth]{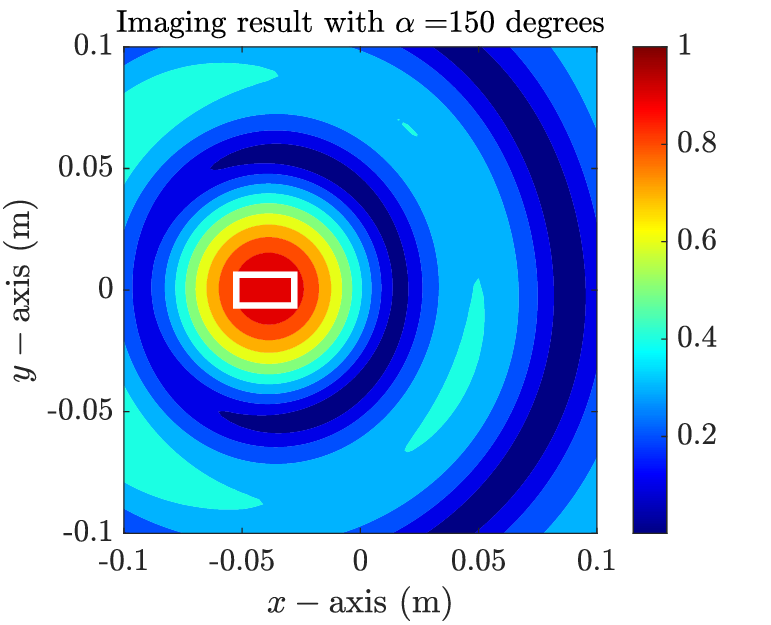}\hfill
\includegraphics[width=.33\textwidth]{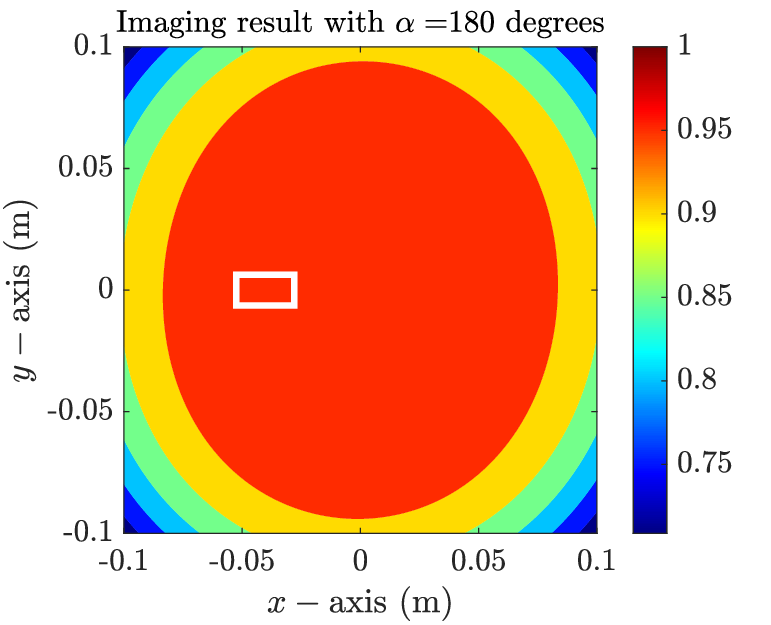}
\caption{\label{Fig3}(Example \ref{Ex3}) Maps of $\mathfrak{F}_{\bfn}(\mx,\alpha)$ at $f=\SI{4}{\GHz}$ with various $\alpha$.}
\end{figure}

\begin{example}[Further result corresponding to the frequency of operation]\label{Ex4}
To examine the imaging performance with respect to frequency changes, we compare the results at $f=\SI{2}{\GHz}$ (Figure \ref{Fig4}) and $f=\SI{6}{\GHz}$ (Figure \ref{Fig5}) in the presence of two dielectric circular inhomogeneities $D_1$ and $D_2$ in Example \ref{Ex2}. In the case of $f=\SI{2}{\GHz}$, the results show lower resolution and wider spreading compared to the $f=\SI{4}{\GHz}$ (Example \ref{Ex2} and Fig. \ref{Fig2}). It is interesting to examine that opposite to the results in Fig. \ref{Fig2}, it is possible to identify $D_1$ and $D_2$ when $\alpha=\SI{60}{\degree}$ and $\alpha=\SI{90}{\degree}$ while impossible to recognize the existence of $D_1$ and $D_2$ when $\alpha=\SI{135}{\degree}$. In the case of $f=\SI{6}{\GHz}$, we can examine similar phenomenon: the identification is significantly dependent on the frequency of operation and bistatic angle $\alpha$.
\end{example}

\begin{figure}[h]
\includegraphics[width=.33\textwidth]{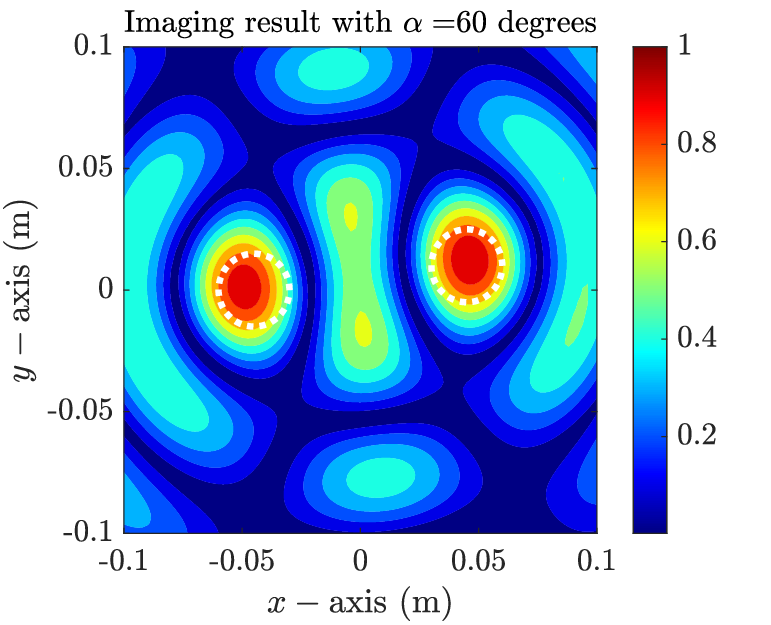}\hfill
\includegraphics[width=.33\textwidth]{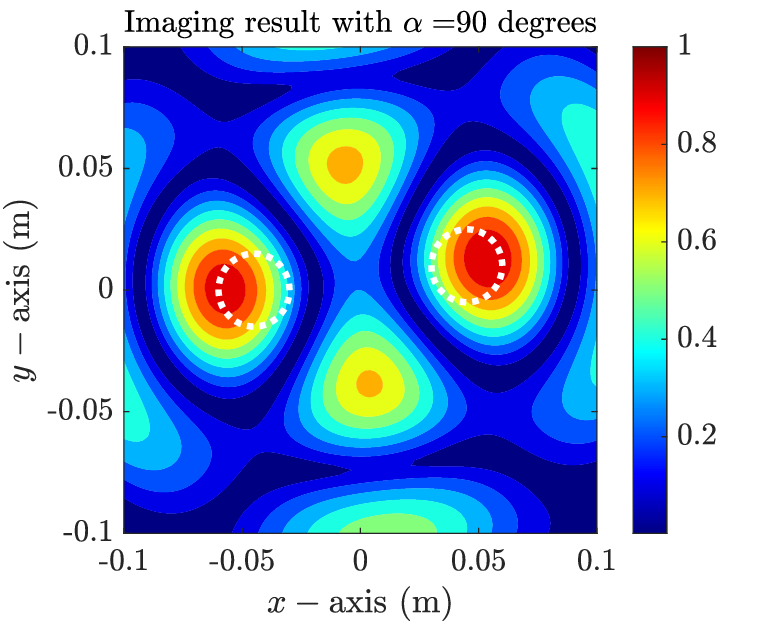}\hfill
\includegraphics[width=.33\textwidth]{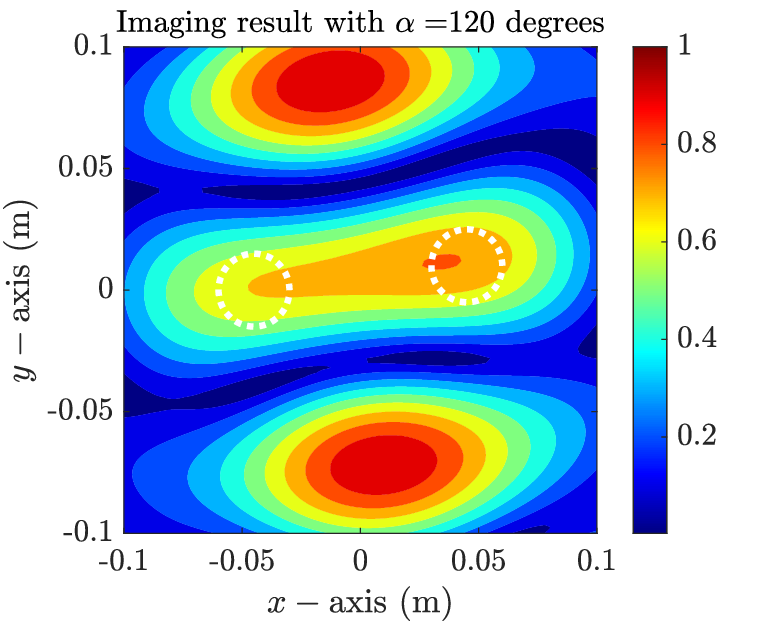}\\
\includegraphics[width=.33\textwidth]{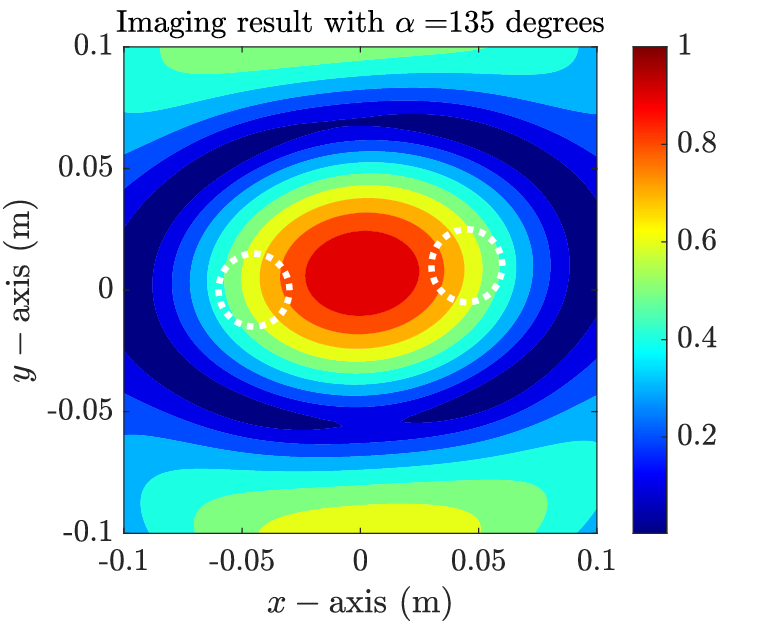}\hfill
\includegraphics[width=.33\textwidth]{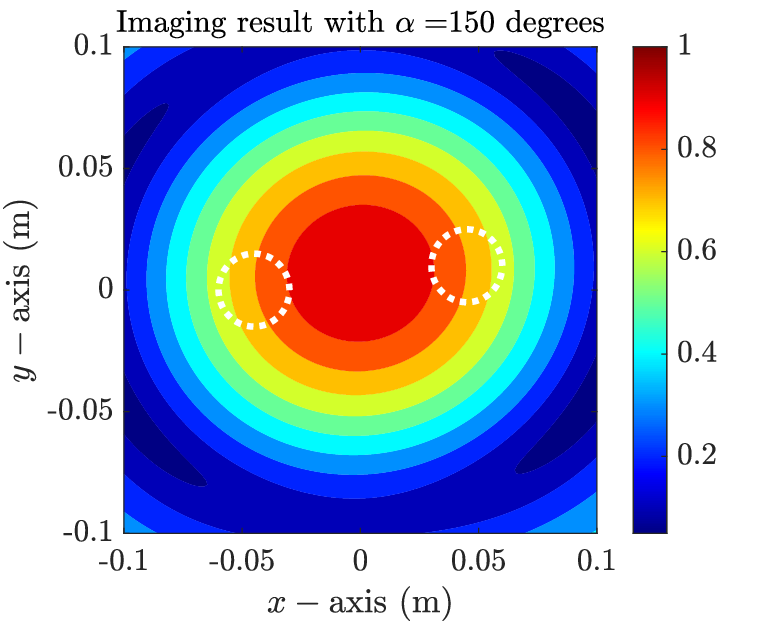}\hfill
\includegraphics[width=.33\textwidth]{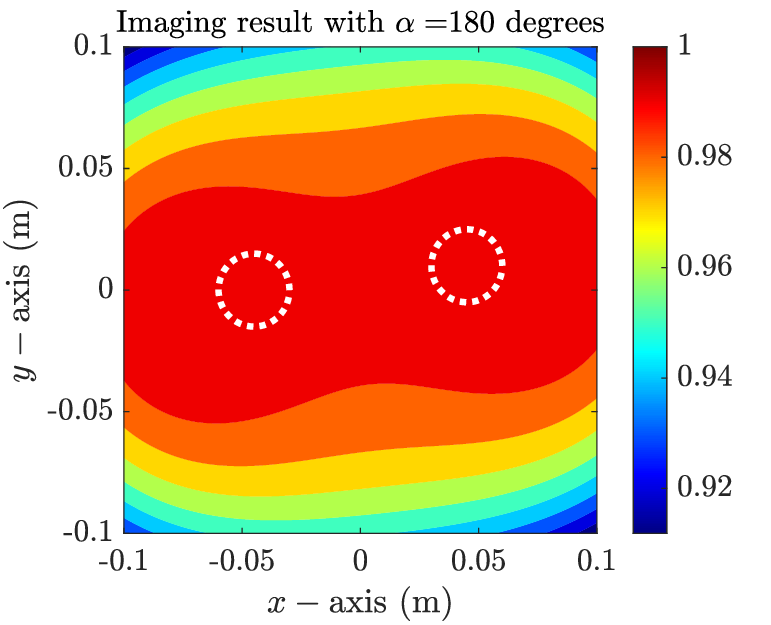}
\caption{\label{Fig4}(Example \ref{Ex4}) Maps of $\mathfrak{F}_{\bfn}(\mx,\alpha)$ at $f=\SI{2}{\GHz}$ with various $\alpha$.}
\end{figure}

\begin{figure}[h]
\includegraphics[width=.33\textwidth]{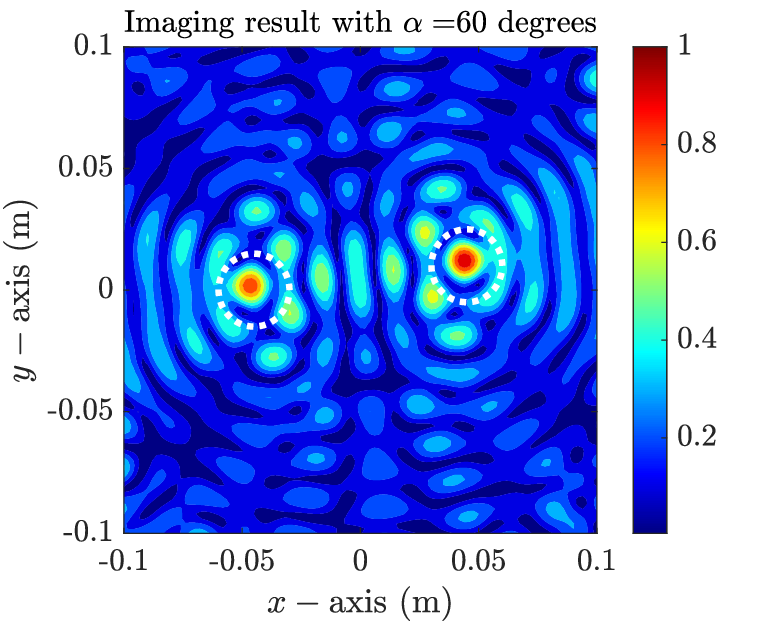}\hfill
\includegraphics[width=.33\textwidth]{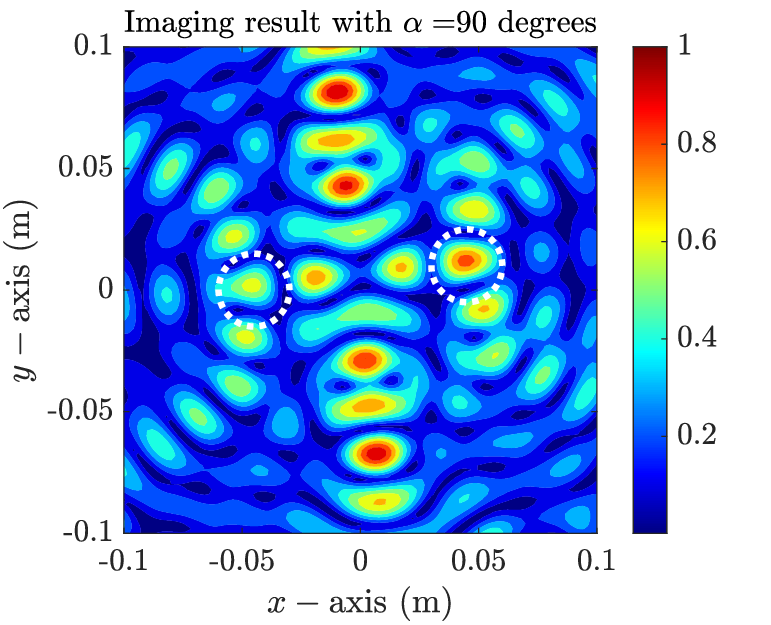}\hfill
\includegraphics[width=.33\textwidth]{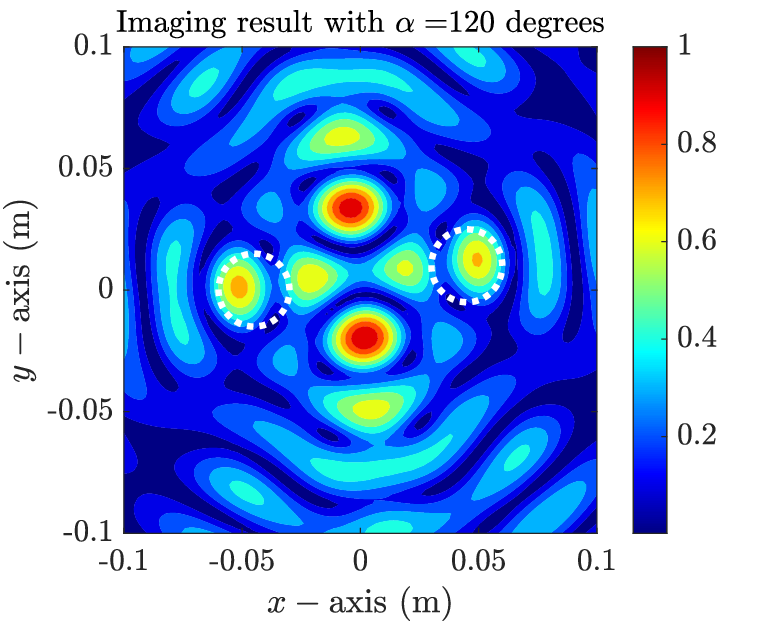}\\
\includegraphics[width=.33\textwidth]{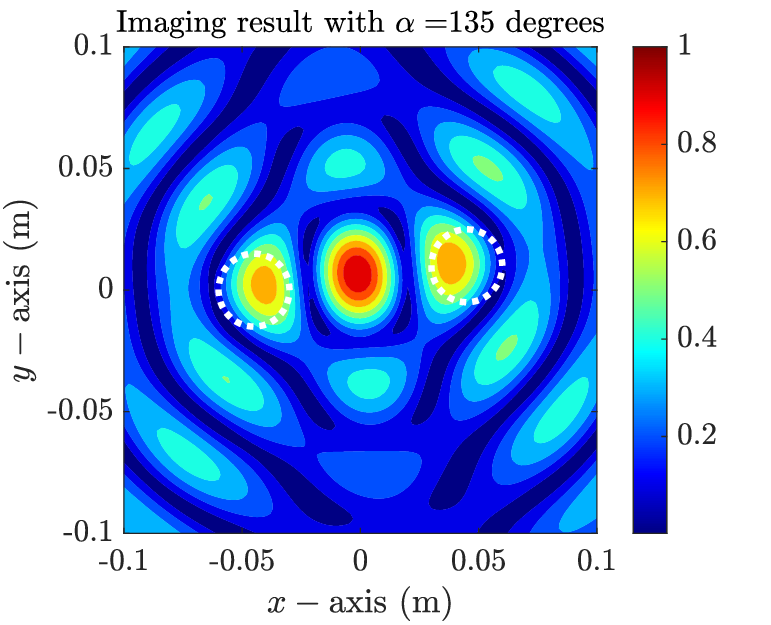}\hfill
\includegraphics[width=.33\textwidth]{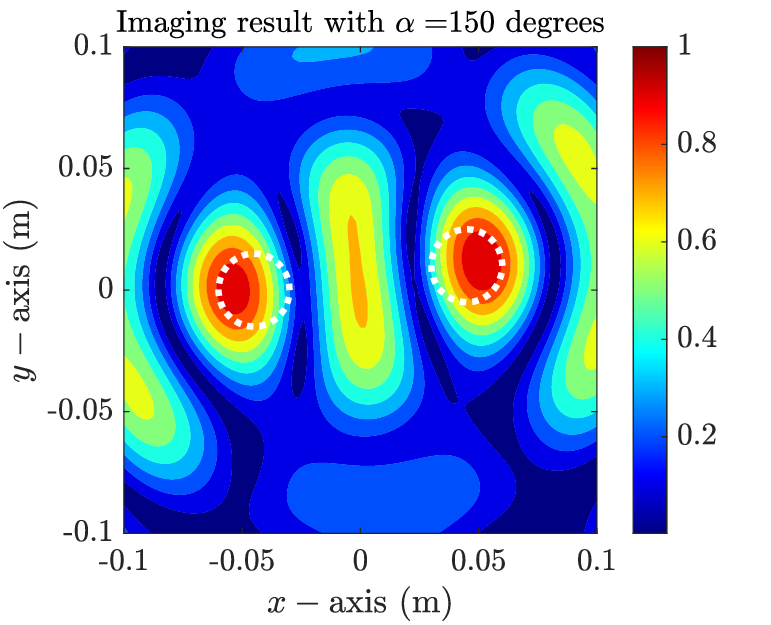}\hfill
\includegraphics[width=.33\textwidth]{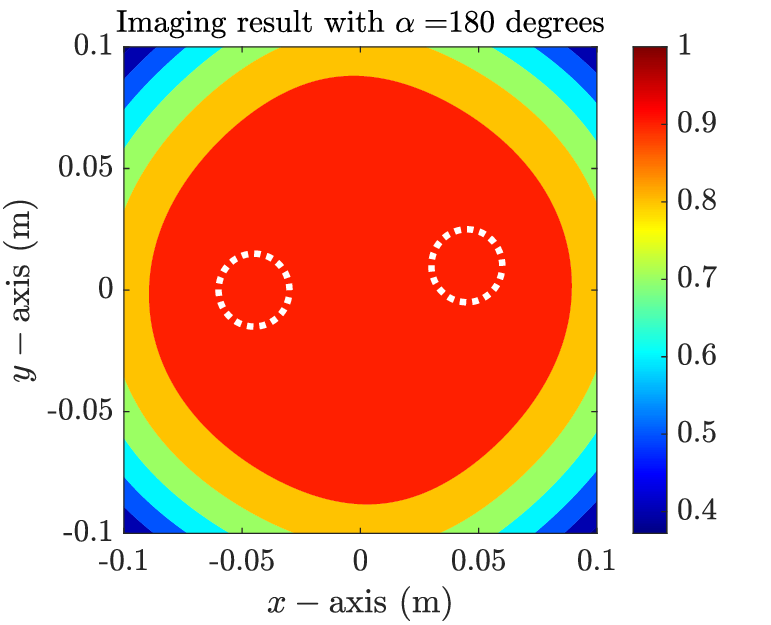}
\caption{\label{Fig5}(Example \ref{Ex4}) Maps of $\mathfrak{F}_{\bfn}(\mx,\alpha)$ at $f=\SI{6}{\GHz}$ with various $\alpha$.}
\end{figure}

\section{Conclusion}\label{sec:5}
In this paper, we considered a bifocusing-based imaging algorithm for the rapid identification of small penetrable dielectric inhomogeneities in a bistatic measurement setup. To demonstrate its applicability, dependence on the bistatic angle, and limitations, we derived the mathematical structure of the indicator function by establishing a relationship involving the infinite series of Bessel functions, the characteristics of the inhomogeneities, and the bistatic angle. We further theoretically verified why the selection of a bistatic angle $\alpha=\SI{180}{\degree}$ is unsuitable for identifying inhomogeneities. Numerical simulation results using the Fresnel experimental dataset polluted by random noise were presented to support the theoretical findings.

Based on the simulation results, the proposed indicator function yields reliable results in the presence of a single target; however, its performance degrades in the presence of multiple targets. The development of an improved imaging algorithm to address this issue remains a subject for future research. Finally, the extension of this work to three-dimensional problems is also a promising research topic.

\backmatter

\bmhead{Acknowledgements}
Part of this work was conducted during a visit to the Department of Mathematics at Hokkaido University. The author would like to thank Professor Nakamura Gen for his kind invitation and the warm hospitality during the visit.

\bmhead{Funding} This work was supported by the National Research Foundation of Korea (NRF) grant funded by the Korea government (MSIT) (RS-2025-16067902).
\bmhead{Author Contribution} The author confirms sole responsibility for the manuscript.
\bmhead{Data Availability Statements} This manuscript has no data from repository.
\bmhead{Conflict of Interest} The author declares that he has no competing interests.


\end{document}